\documentclass{article}
\usepackage{fullpage}
\usepackage{amsmath}
\usepackage{amsthm}
\usepackage{amssymb}
\usepackage{url}
\usepackage{graphicx}
\usepackage{verbatim}
\usepackage{listings}
\usepackage{url}
\usepackage{graphicx}
\usepackage{tikz}
\usetikzlibrary{arrows,chains,matrix,positioning,scopes}
\usetikzlibrary{arrows,shapes,positioning}
\usetikzlibrary{decorations.markings}
\tikzstyle arrowstyle=[scale=1]
\tikzstyle directed=[postaction={decorate,decoration={markings,
    mark=at position .65 with {\arrow[arrowstyle]{stealth}}}}]
\tikzstyle reverse directed=[postaction={decorate,decoration={markings,
    mark=at position .65 with {\arrowreversed[arrowstyle]{stealth};}}}]

\newtheorem{thm}{Theorem}
\newtheorem{biog}{}
\newtheorem{affil}{}
\newtheorem{lemma}{Lemma}
\newtheorem{prop}{Proposition}
\theoremstyle{definition}
\newtheorem{definition}{Definition}
\newtheorem{remark}{Remark}
\newtheorem{example}{Example}

\begin{document}

\title{A Necessary and Sufficient Condition for Local Maxima of Polynomial Modulus Over Unit Disc}

\author{Bahman Kalantari \\
Department of Computer Science, Rutgers University, NJ\\
kalantari@cs.rutgers.edu
}
\date{}
\maketitle

\begin{abstract}
An important quantity associated with a complex polynomial $p(z)$ is $\Vert p \Vert_\infty$, the maximum of its modulus over the unit disc $D$.  We prove, $z_* \in D$ is a local maximum of $|p(z)|$ if and only if $a_*$ satisfies, $z_*=p(z_*)|p'(z_*)|/p'(z_*)|p(z_*)|$, i.e. it is proportional to its corresponding Newton direction.  This explicit formula gives rise to novel iterative algorithms for computing $\Vert p \Vert_\infty$.  We describe two such algorithms, including a Newton-like method and  present some visualization of their performance.
\end{abstract}


\section{Introduction}
For a complex number $z=x+ iy$,  its conjugate and modulus are $\overline z=z-iy$ and $|z|=\sqrt{z \overline z}=\sqrt{x^2+y^2}$, respectively. Given a complex polynomial $p(z)$, the maximum  of $|p(z)|$ on the unit disc $D=\{z: \quad |z| \leq 1\}$, denoted by $\Vert p \Vert_\infty$, is an important quantity associated with $p(z)$. This quantity can also be used to bound some other attributes and norms of a complex polynomial, see e.g.  \cite{Bor}, \cite{Green}.  In this article we study the optimization of the modulus of a polynomial over the unit disc and offer new algorithms to compute  $\Vert p \Vert_\infty$.

By the {\it maximum modulus principle}, see e.g  \cite{Bak1},  \cite{Kal},  $\Vert p \Vert_\infty$ is attained at a boundary point of $D$, i.e.

\begin{equation} \label{eq1}
\Vert p \Vert_\infty = \max \big \{ |p(z)|: \quad z \in D,  \quad  |z|=1 \big \}.
\end{equation}

From (\ref{eq1}) and the polar form of complex numbers we have,
\begin{equation} \label{eq2}
\Vert p \Vert_\infty = \max \big \{q(t)=|p(e^{it})|: \quad t \in [0,2\pi] \big \}.
\end{equation}

Green \cite{Green} makes use of (\ref{eq2}) and a lemma due to Ste\v{c}kin that gives lower bounds to $q(t)$ on subintervals of $[0,2\pi]$ to describe a partitioning algorithm  for computing $\Vert p \Vert_\infty$. For each subinterval the modulus is evaluated at the midpoint.  This allows discarding a subinterval, or demands further partitioning.  Green's algorithm may require many polynomial evaluations due to the possibility of considering a large number of subintervals.

In this article we derive an explicit formula as a necessary and sufficient condition for any local maximum of $|p(z)|$ over $D$.  In particular, any $z_*$ for which $\Vert p \Vert_\infty=|p(z_*)|$ must necessarily satisfy the formula. The formula gives an explicit connection between  a local maximum $z_*$ on the unit circle and the negation of Newton direction, $p(z_*)/p'(z_*)$, at that point. Clearly, not only such an explicit formula is interesting from the theoretical point of view, but from the algorithmic point of view as well.
The derivation of the formula is based on a property of polynomials referred as {\it Geometric Modulus Principle} proved in \cite{Kal}, a result that characterizes the cone of ascent and descent directions of polynomial modulus at an arbitrary point in the complex plane. Having derived this formula, we then describe some iterative algorithms that are inspired by it, in particular Pseudo-Newton  method.   Finally, we demonstrate some visualization of the  performance of these algorithms via images for some specific polynomials.  These images are visually reminiscent of those derived from {\it polynomiography}, a term used for algorithmic visualization in solving a polynomial equation, see \cite{Kalan}. Conceptually too they may fit well in that category.  Polynomiograpy images are not necessarily fractal in nature.

In Section 1, we review the Geometric Modulus Principle in \cite{Kal}. In Section 2, we derive the main formula.  In Section 3, we describe some relevant algorithms based on the formula and present visualizations for some specific polynomials. We conclude with some remarks.

\section{Review of Geometric Modulus Principle}

\begin{definition}  Consider a complex polynomial $p(z)$ at a given  complex number $z_0$. Given a real number $\theta$, if there exists $t_* >0$ such that for all $t \in (0, t_*)$,  $|p(z_0+ t e^{i \theta})| > |p(z_0)|$, we call  $e^{i \theta}$ an {\it ascent direction}  for $|p(z)|$ at  $z_0$. We refer to $\theta$ as an {\it argument of ascent} and any $t \in (0, t_*)$ as a {\it step-size}. We refer to the collection of all ascent directions as the {\it cone of ascent} at $z_0$.  The {\it cone of descent} is defined analogously.
\end{definition}

Ascent and descent directions centered at $z_0$ can be used to give a partitioning of a disc of certain radius  centered at $z_0$. Since we are interested in the shape of the cone as opposed to the radius, we may consider the unit disc at $z_0$.  The following theorem reveals the geometric property of this partitioning (see Figure \ref{Fig1}).

\begin{thm} \label{thm1} {\bf (Geometric Modulus Principle) \cite{Kal}} Let $p(z)$ be a nonconstant polynomial. If $p(z_0)=0$, then every direction at $z_0$ is an ascent direction for $|p(z)|$. If $p(z_0) \not =0$, then the cones of ascent and descent directions at $z_0$ partition the unit disc centered at $z_0$ into alternating ascent and descent sectors of equal angle $\pi/k$, where $k \geq 1$ is the smallest index with $p^{(k)}(z_0) \not=0$. Specifically, if
$\overline {p(z_0)} p^{(k)}(z_0)=re^{{i} \alpha}$ then $e^{i\theta}$ is an ascent direction if $\theta$ satisfies
\begin{equation}
\frac{2N \pi - \alpha}{k}- \frac{\pi}{2k} < \theta < \frac{2N \pi -\alpha}{k}+ \frac{\pi}{2k},
\end{equation}
where $N$ is an integer. The ascent directions thus define $k$ sectors in the unit disc, the boundary lines excluded.  The interiors of the remaining sectors in the unit disc form the sectors of descent. A boundary line between two sectors is either an ascent direction or a descent direction. $\square$
\end{thm}

\begin{figure}[htpb]
	\centering
	
	\begin{tikzpicture}[scale=.5]
		\begin{scope}[black]
		 \draw  (0.0,0.0) circle (3.);
		 \draw (7.0,0.0) circle (3.);
		 \draw (14.0,0.0) circle (3.);
         \draw (21.0,0.0) circle (3.);
\filldraw (0,0) circle (2pt);
\filldraw (7,0) circle (2pt);
\filldraw (14,0) circle (2pt);
\filldraw (21,0) circle (2pt);
\end{scope}
\begin{scope}[black]
		 \clip  (0.0,0.0) circle (3.);
\fill[color=gray!39] (-10, 10) rectangle (10, -10);
\filldraw (0,0) circle (2pt);
\end{scope}

\begin{scope}[black]
		 \clip  (7.0,0.0) circle (3.);
\clip (4.0,0.0) -- (10.0,.0) -- (10.0,3.0) -- (4.0, 3.0)-- cycle;
\fill[color=gray!39] (-20, 20) rectangle (20, -20);
\filldraw (7,0) circle (2pt);
\end{scope}

\begin{scope}[black]
		 \clip  (14.0,0.0) circle (3.);
\clip (14.0,0.0) -- (17.0,.0) -- (17.0,3.0) -- (14.0, 3.0)-- cycle;
\fill[color=gray!39] (-20, 20) rectangle (20, -20);
\filldraw (14,0) circle (2pt);
\end{scope}
\begin{scope}[black]
\clip  (14.0,0.0) circle (3.);
\clip (11.0,0.0) -- (14.0,.0) -- (14.0,-3.0) -- (11.0, -3.0)-- cycle;
\fill[color=gray!39] (-20, 20) rectangle (20, -20);
\filldraw (14,0) circle (2pt);
\end{scope}

\begin{scope}[black]
\clip  (21.0,0.0) circle (3.);
\clip (21.0,0.0) -- (27.0,0.0)  -- (24,-5.196)-- cycle;
\fill[color=gray!39] (-30, 30) rectangle (30, -30);
\filldraw (21,0) circle (2pt);
\end{scope}

\begin{scope}[black]
\clip  (21.0,0.0) circle (3.);
\clip (21.0,0.0) -- (15.0,0.0)  -- (18,-5.196)-- cycle;
\fill[color=gray!39] (-30, 30) rectangle (30, -30);
\filldraw (21,0) circle (2pt);
\end{scope}

\begin{scope}[black]
\clip  (21.0,0.0) circle (3.);
\clip (21.0,0.0) -- (24.0,5.196)  -- (18,5.196)-- cycle;
\fill[color=gray!39] (-30, 30) rectangle (30, -30);
\filldraw (21,0) circle (2pt);
\end{scope}

\end{tikzpicture}
	\caption{From left to right: Sectors of ascent (gray) and descent (white) at $z_0$ for $p(z_0)=0$, and $p(z_0) \not =0$, $k=1,2,3$.}
\label{Fig1}
\end{figure}
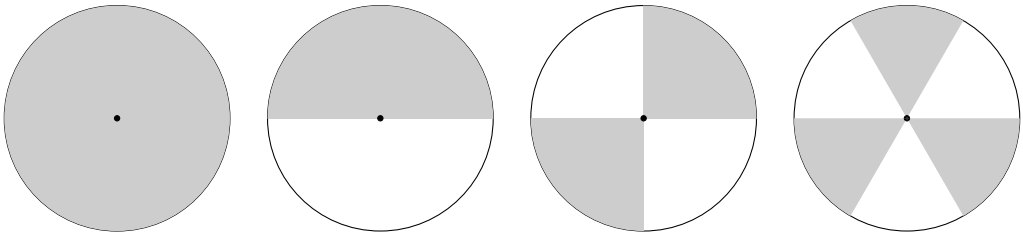

\section{A Necessary and Sufficient Condition for Local Maxima}

In this section we state and prove our main result.

\begin{thm} \label{thm2} Let $p(z)$ be a nonconstant polynomial. Let $D=\{z: |z| \leq 1\}$.  A point $z_* \in D$ is a local maximum of $|p(z)|$ over $D$ if and only if
\begin{equation} \label{eq3}
z_*= \bigg({\frac{p(z_*)}{p'(z_*)}} \bigg ) \bigg / \bigg ( \bigg |{\frac{p(z_*)}{p'(z_*)}} \bigg | \bigg ).
\end{equation}
\end{thm}

\begin{proof}
From Theorem \ref{thm1} it easily  follows  that any local maximum of $z_*$ of $|p(z)|$ in $D$ must lie on the boundary of $D$.  In fact Theorem \ref{thm1} implies the maximum modulus principle for polynomials.  Clearly, $p(z_*) \not= 0$.  Furthermore, from Theorem \ref{thm1} it must be the case that $p'(z_*) \not =0$ since otherwise the cone of ascent and descent directions at $z_*$ would split into four or more sectors (see Figure \ref{Fig1} the images on the right) and this would contradict that $z_*$ is a local maximum. Thus the index $k$ in Theorem \ref{thm1} must equal $1$.

It follows that the only possibility for the cone of ascent at $z_*$ is as shown in Figure \ref{Fig2}, i.e. the boundary line between the cones of ascent and descent at $z_*$ must be tangential to the unit disc, and the cone of ascent (the gray area) must be pointing outward as shown in the figure.  Otherwise, $z_*$ would not be a local maximum of $|p(z)|$ over $D$.

Since $k=1$, from Theorem \ref{thm1}, $\overline {p(z_*)} p'(z_*)=r e^{i \alpha}$ and the cone of ascent is the set of all directions $e^{i\theta}$ satisfying
\begin{equation} \label{eqinterval}
- \frac{\pi}{2}- \alpha < \theta < \frac{\pi}{2} - \alpha.
\end{equation}
This can only happen  (see Figure \ref{Fig2}) if we have $\overline z_* =e^{i \alpha}$, i.e.
\begin{equation} \label{eqx}
\overline {p(z_*)} p'(z_*)= r \overline z_*.
\end{equation}
Taking the conjugate of both sides in (\ref{eqx}) we get
\begin{equation} \label{eqx1}
p(z_*) \overline {p'(z_*)} = r z_*.
\end{equation}
Since $|z_*|=1$, from (\ref{eqx1}) we have $r= |p(z_*) \overline {p'(z_*)}|$. Substituting this in the above proves the main formula, (\ref{eq3}).

Conversely, suppose that  $z_*$ satisfies (\ref{eq3}).  Then from Theorem \ref{thm1} the cone of ascent at $z_*$ must be as in Figure \ref{Fig2}. But this implies $z_*$ must be a local maximum.
\end{proof}

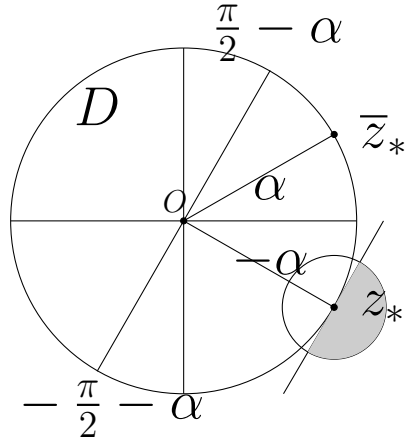
\begin{figure}[htpb]
	\centering
	
	\begin{tikzpicture}[scale=2.3]
		\begin{scope}[black]
		 \draw  (0.0,0.0) circle (1);
		 \draw (.87,-0.5) circle (.3);
\filldraw (.87,-0.5) circle (.5pt);
\filldraw (0,0) circle (.5pt);
\draw (0.5,0.87) -- (-.5,-.87);
\draw (1,0) -- (-1,0);
\draw (0,1) -- (0,-1);
\draw (.87,-0.5) -- (0,0);
\draw (.87,0.5) -- (0,0);
\draw (.5,.07)  node[above] {\huge{$\alpha$}};
\draw (.5,-.07)  node[below] {\huge{$-\alpha$}};
\draw (-0.05,0)  node[above] {\large{$O$}};
\draw (-0.5,.5)  node[above] {\huge{$D$}};
\end{scope}

\begin{scope}
\draw  (0.87,-0.5) --(1.155,0);
\draw  (0.87,-0.5) --(.577, -1);
\end{scope}

\begin{scope}[black]
		 \clip  (.87,-0.5) circle (.3);
\clip (1.155,0) -- (.0,-2) -- (10.0,-2) -- (1.55,0)--cycle;
\fill[color=gray!39] (-40, 40) rectangle (40, -40);
\end{scope}
\filldraw (.87,-0.5) circle (.5pt) node[right] {\huge{$~z_*$}};
\filldraw (.87,0.5) circle (.5pt) node[right] {\huge{$~\overline{z}_*$}};
\draw (0.5,0.87) node[above] {\huge{$~\frac{\pi}{2}-\alpha$}};
\draw (-.5,-.87) node[below] {\huge{$~-\frac{\pi}{2}-\alpha$}};
\end{tikzpicture}
	\caption{The shape of cone ascent at a local max $z_*$ and its relationship with $\alpha$.}
\label{Fig2}
\end{figure}

\section{Algorithmic Applications}

The explicit formula  (\ref{eq3}) has several applications. Firstly, it is a convenient formula to efficiently test if a given $z$ on the boundary of $D$ is a local maximum of $|p(z)|$ on $D$.  It only requires evaluation of $p(z)$ and its derivative. Additionally,  by evaluating the gap between $z$ and the corresponding right-hand-side of (\ref{eq3}) we get a measure of how close the point is to being a local maximum. A more important application of the formula is in deriving algorithms for computing $\Vert p \Vert_\infty$.   We will describe two such algorithms next and also give examples of their behavior.

\subsection{Direct Fixed Point Iterations}

From Theorem \ref{thm2}, $z_*$ is a local maximum of $|p(z)|$ over $D$ if and only if it is a fixed point of
\begin{equation} \label{eq4}
F(z)= \bigg( \frac{p(z)}{p'(z)} \bigg) \bigg / \bigg ( \bigg | \frac{p(z)}{p'(z)} \bigg |\bigg ).
\end{equation}
A straightforward algorithm is to apply the fixed point iteration to $F(z)$: Given a seed  $z_0 \in \mathbb{C}$, set
\begin{equation} \label{eq5}
z_{k+1}= F(z_k), \quad k =0, 1, \dots
\end{equation}
The analysis of convergence behavior is however nontrivial. First we give a definition.

\begin{definition}  \label{deffix}

(i) The {\it basins of attraction} of a fixed point $z_*$ of $F(z)$ is the set of points $z_0 \in \mathbb{C}$ whose {\it orbit}, $\{F(z_k): k=0,1, \dots\}$,  converge to $z_*$.

(ii) We say a fixed point $z_*$ is {\it attractive} if there is a neighborhood $N$ of $z_*$ such that for any $z \in N$ the fixed point iteration converges to $z_*$.

(iii) We say  a fixed point $z_*$ is {\it repulsive} if there exists a neighborhood $N$ or $z_*$ such that $N \cap F(N) = \{z_*\}$.  In other words, excluding $z_*$, the image of $N$ under $F$  is moved outside of $N$.

(iv) We say a fixed point $z_*$ is {\it indifferent} if it is not attractive but the basin of attraction of $z_*$ contains an open set. \end{definition}

For a rational function, i.e. the quotient of two polynomials, the characterization of fixed points is simply in terms of the modulus of its derivative at the fixed point (see e.g. Beardon \cite{Bea}).  However, $F(z)$ in (\ref{eq4})
is not differentiable at a fixed point.  Analogous to the visual study of typical iterative methods  such as Newton method,  we can consider the performance of the iterations of $F(z)$, say in a square containing the unit disc $D$ and based on the behavior of their orbits associate color coding to points in this square. Figure \ref{Fig3}  gives a color coding for a few cases to be described in the example below.

\begin{example}  Consider $p(z)=z^n-1$, $n >1$.  The modulus  of $p(z)$ is maximum at $z_*$ if and only if $z_*^n=-1$.  For $n=2$, the fixed points of $F(z)$ are $\pm i$.  For $n=3$, the fixed points are the negation of the cube-roots of unity.  For $n=4$, the fixed points are $\pm \sqrt{2}/2 \pm i \sqrt{2}/2$. For $n=5$, the fixed points are the negation of the fifth roots of unity.

Figure \ref{Fig3} gives a visualization of the fixed point iterations of $F(z)$ in the square $[-1.5, 1.5] \times [-1.5, 1.5]$ and a coloring of the basins of attraction. As we see these basins are  well-behaved for $n=2,3$, somewhat analogous to the basins of attraction of iterative methods such as Newton method for solving quadratic and cubic polynomials. However, these images are not fractal as would Newton's method be for $z^3-1$.
For analysis  and visualization in solving polynomial equations via iterative methods, called {\it polynomiography}, see \cite{Kalan}.

The images in Figure \ref{Fig3} show  the fixed points of $F(z)$ and the roots of $p(z)=z^n-1$.  For $n=2,3$ the fixed points are attractive.
The sword-like area in all four images represent a fast-convergence area. For $n=4,5$ the situation is different. The fixed points are not attractive, however for each fixed point there is small region of point that forms its basin of attraction (the sword-like area).  In other words,  according to definition (\ref{deffix}), the fixed points are not repulsive but appear to be indifferent. For $n=5$ the orbit is shown near a point that starts outside of the  unit circle and becomes chaotic after the first iteration.
\end{example}

\begin{figure}[h!]
\centering
\includegraphics[width=1.5in]{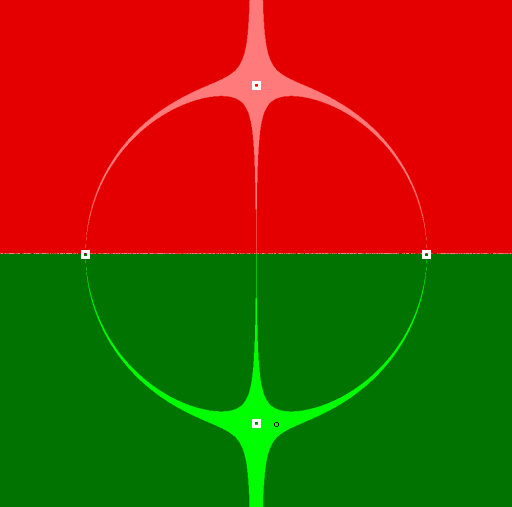}
\includegraphics[width=1.5in]{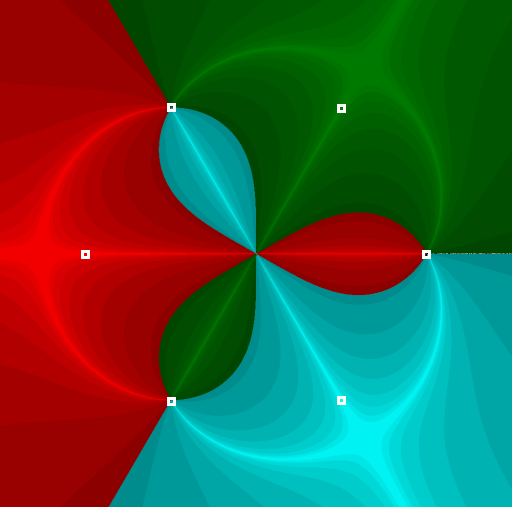}
\includegraphics[width=1.5in]{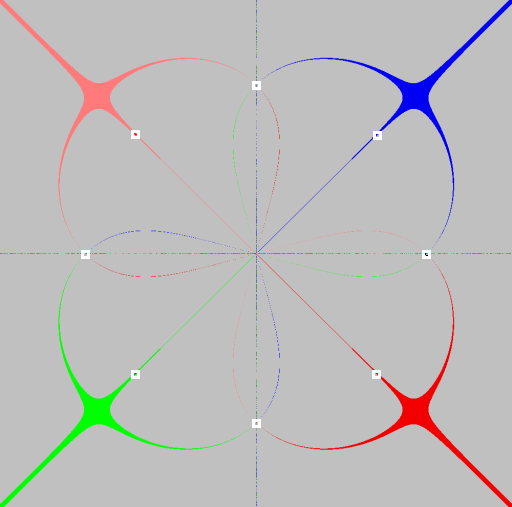}
\includegraphics[width=1.5in]{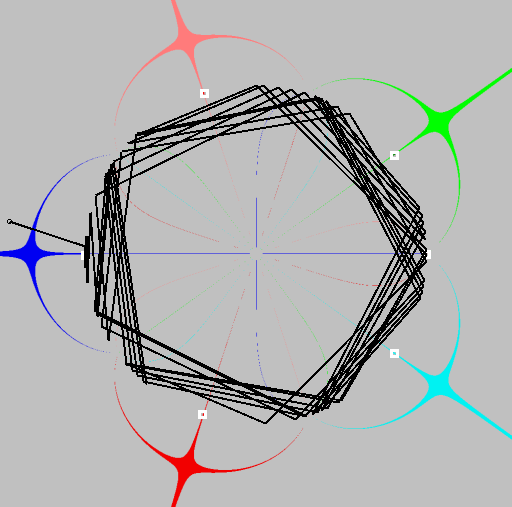}
\caption{Basins of attraction of fixed points of $F(z)$ for $p(z)=z^n-1$, $n=2-5$.} \label{Fig3}
\end{figure}

\subsection{Solving a Pseudo-Polynomial Equation Via a Pseudo-Newton Method}

Rather than solving $F(z)=z$ via fixed point iteration (\ref{eq5}), consider solving the following {\it pseudo-polynomial} equation
\begin{equation} \label{eqg}
G(z)=0,
\end{equation}
where
\begin{equation} \label{eqgz}
G(z)= p(z)|p'(z)|-z p'(z)|p(z)|.
\end{equation}
Clearly, $G(z_*)=0$ if and only if either $z_*$ is a fixed point of $F(z)$ (i.e. a local maximum of $|p(z)|$ over $D$), or $z_*$ is a zero of the product $p(z) \cdot p'(z)$.  To solve (\ref{eqg}), we consider a Newton-like methods, we call  {\it Pseudo-Newton} method:  Pick a seed $z_0 \in \mathbb{C}$.  Then given an iterate $z_k$,  apply one iteration of Newton method to solve
\begin{equation}
G_k(z)= p(z)|p'(z_k)|-z p'(z)|p(z_k)|=0.
\end{equation}
Specifically,  set
\begin{equation} \label{eqgk}
z_{k+1}=z_k- \frac{G_k(z_k)}{G'_k(z_k)}.
\end{equation}
Replace $z_k$ with $z_{k+1}$ and repeat the process.

Polynomiography of this Pseudo-Newton scheme is given in Figure \ref{Fig4}. We give some explanations and make a few observations.  Each fixed point of $F(z)$ is a zero of $G(z)$. From the figure the roots of $p(z)$ or $p'(z)$ appear to be repulsive fixed points of the Pseudo-Newton method.  Curiously,  near a root of $p(z)$ there is a small red neighborhood that seems to belong to the basin of attraction of  the Pseudo-Newton method corresponding to a fixed point of $F(z)$.  If this is valid for a general polynomial $p(z)$ it would suggest that to compute $\Vert p \Vert_\infty$ we would compute the roots of $p(z)$ and iterate the Pseudo-Newton method in a small neighborhoods of the roots of $p(z)$.

These simple examples reveal interesting behavior in the dynamic of fixed point iterations of $F(z)$,  and those corresponding to the Pseudo-Newton method for solving $G(z)=0$.   The study of convergence and rate of convergence of this Pseudo-Newton method is a subject of future research.  In fact it is possible to define more general iterative methods like the Pseudo-Newton method. In particular, the {\it Basic Family} is an infinite family of iteration functions for polynomial root-finding, $\{B_m(z): m=2, 3, \dots\}$, studied thoroughly, see \cite{Kalan}. While $B_2(z)$ is Newton's method,  at a simple roots of a polynomial the order of convergence of $B_m(z)$ is $m$. We can thus solve $G(z)=0$ by making use of any member of the Basic Family:  Given $z_k$ we compute
\begin{equation}
z_{k+1}= B_m(z_k),
\end{equation}
as applied to $G_k(z)$, then replace $z_k$ by $z_{k+1}$ and repeat.  Not only we would expect faster convergence than the case of Pseudo-Newton method but interesting polynomiography.  In forthcoming work we will carry out further  experimentation in solving $G(z)=0$ via Basic Family, as well as investigating their theoretical performance.

\begin{figure}[h!]
\centering
\includegraphics[width=1.5in]{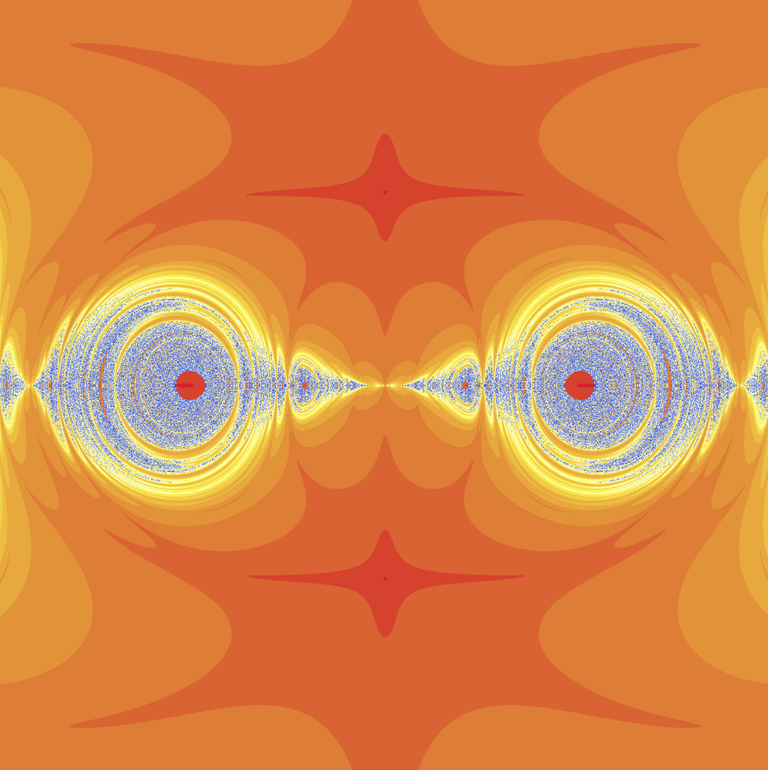}
\includegraphics[width=1.5in]{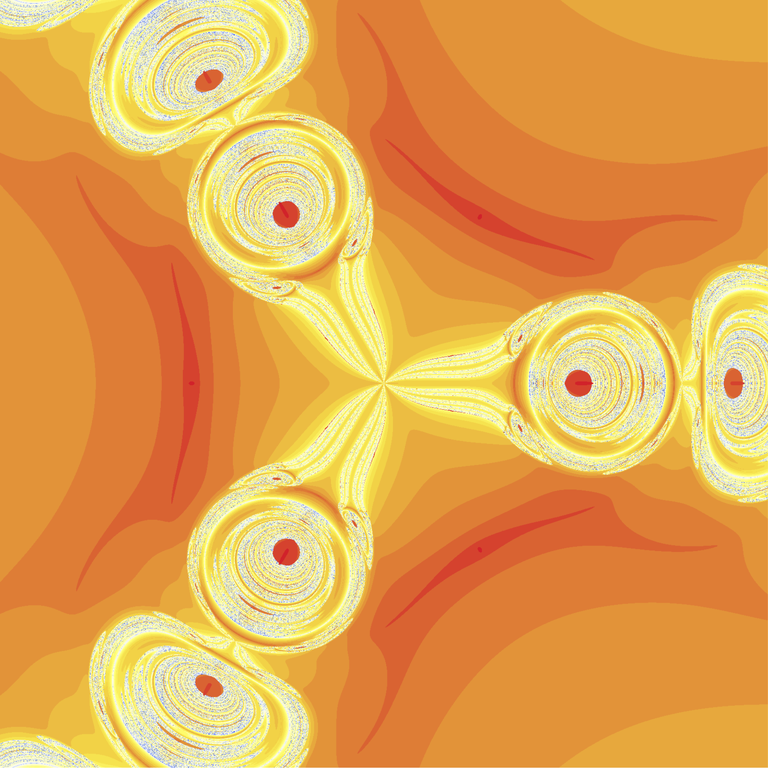}
\includegraphics[width=1.5in]{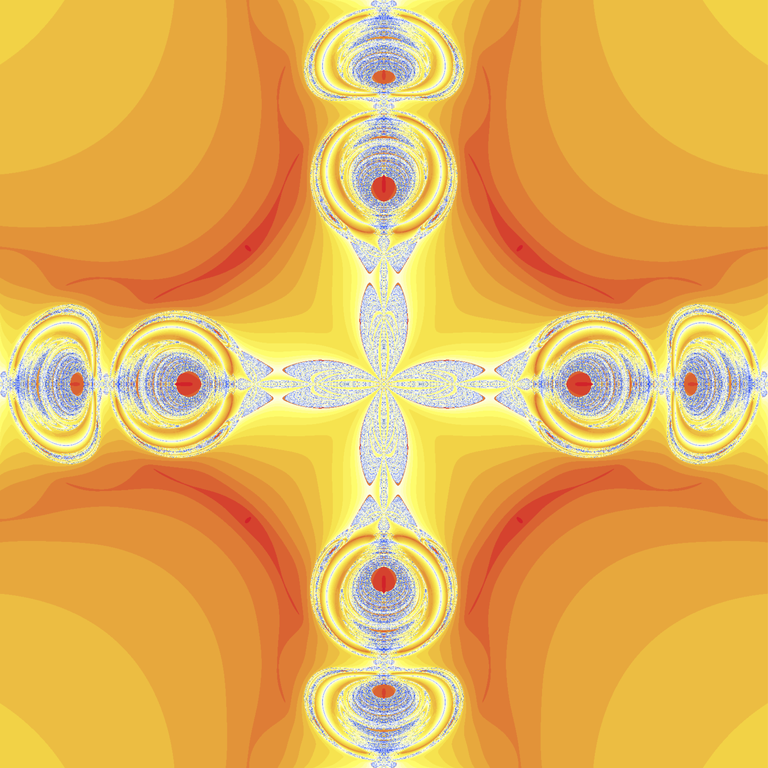}
\caption{Polynomiography of Pseudo-Newton for $p(z)=z^3-1$, $n=2,3,4$.} \label{Fig4}
\end{figure}

\section{Concluding Remarks} In this article we have derived a necessary and sufficient formula for  any local maximum of the modulus of a polynomial over the unit disc. Moreover, inspired by the formula we have described two iterative methods for computing a local maximum and $\Vert p \Vert_\infty$. We anticipate that these will result in  practical methods for approximating $|p|$, as well as leading to new theoretical  findings. These are supported by the few examples considered, in particular those based on pseudo-Newton iterations.

According to Bernstein's inequality (see e.g. \cite{Bernstein}), $\Vert p' \Vert_\infty \leq n \Vert p \Vert_\infty$. This implies if $z^* \in D$ satisfies $|p(z^*)|= \Vert p \Vert_\infty$, then
\begin{equation}
\bigg | \frac{p(z^*)}{p'(z^*)} \bigg | \geq \frac{1}{n}.
\end{equation}
Such information may become useful in the search for $z^*$.  Suppose that the ratio  $|p(z^*)|/|p'(z^*)|$ is known. Then  (\ref{eq3}) implies that computing $z^*$ amounts to solving a single polynomial equation of degree $n$, and computing a root whose modulus is one. In this sense, computing $\Vert p \Vert_\infty$ is at least as hard as solving the equation $p(z)-rzp'(z)=0$, where $r$ is a some constant satisfying $r \geq 1/n$. If we only know $|p(z^*)|$ but  $|p'(z^*)|$ is unknown the search for $z^*$ is harder.

Finally, our formula for local maxima over the unit disc easily generalizes to discs of radius $\rho$.  The literature on polynomial inequalities also deals with such maximization and their relation to $\Vert p \Vert_\infty$, e.g. see \cite{Qazi}. This also suggests the possibility of using path-following methods for computing $\Vert p \Vert_\infty$. We will explore these in a forthcoming work.

\paragraph{Acknowledgments.} I would like to thank Dr. Fedor Andreev of Western Illinois University for carrying out the computation resulting in images of Figure \ref{Fig4}.

\bigskip


\end{document}